\documentclass{article}
\usepackage[journal=PM,    
lang=british,   
]{ems-journal}
\usepackage{slashed} 
\usepackage{enumerate}
\usepackage{soul}
\usepackage{mathrsfs}
\usepackage{amsbsy}
\usepackage{csquotes}
\usepackage{stmaryrd}
\usepackage{tikz}
\usepackage{tikz-cd}
\usetikzlibrary{matrix}
\usepackage{gensymb}
\usepackage{bm}
\usepackage{tensor}
\hypersetup{%
    pdfborder = {0 0 0}
}
\usepackage{mathtools}
\usepackage{braket}
\usepackage{bbm} 

\usepackage{amsthm}
\theoremstyle{plain}
\newtheorem{theorem}{Theorem}
\newtheorem{lemma}[theorem]{Lemma}

\newtheorem{corollary}[theorem]{Corollary}

\theoremstyle{definition}

\newtheorem{remark}[theorem]{Remark}

\newtheorem{prop/def}[theorem]{Proposition/Definition}

\newcommand*\Z{\mathbb{Z}}

\newcommand*\K{\mathbb{K}}

\newcommand*\CC{\mathbb{C}}
\newcommand*\GG{\mathbb{G}}
\newcommand*\mO{\mathcal{O}}

\newcommand*\ch{\textnormal{ch}}
\newcommand*\mM{\mathcal{M}}

\newcommand*\Quot{\textnormal{Quot}_X}

\newcommand*\QuotX{\textnormal{Quot}_X(E,n)}

\newcommand*\vir{\textnormal{vir}}
\newcommand*\rk{\textnormal{rk}}

\newcommand\numberthis{\addtocounter{equation}{1}\tag{\theequation}}

\newcommand\Item[1][]{%
  \ifx\relax#1\relax  \item \else \item[#1] \fi
  \abovedisplayskip=0pt\abovedisplayshortskip=0pt~\vspace*{-\baselineskip}}
\DeclareRobustCommand\optionalsec[1]{%
  \ifnum\pdfstrcmp{#1}{\thesection}=0\else#1.\fi
}
\numberwithin{equation}{section}

\begin{document}

\title{Application of Lagrange inversion to wall-crossing for Quot schemes on surfaces}
\titlemark{Application of Lagrange inversion to wall-crossing for Quot schemes on surfaces}



\emsauthor{1}{
	\givenname{Arkadij}
	\surname{Bojko}
	\mrid{1448168}
	\orcid{0000-0003-1530-2418}}{A. Bojko}

\Emsaffil{1}{
	\department{Institute of Mathematics}
	\organisation{Academia Sinica}
	\rorid{0120q3031}
	\address{6F, Astronomy-Mathematics Building,
No. 1, Sec. 4, Roosevelt Road}
	\zip{10617}
	\city{Taipei City}
	\country{Taiwan}
	\affemail{abojko@gate.sinica.edu.tw}}

\classification[05A15]{13F25}


\begin{abstract}
Motivated by my work on enumerative invariants for Quot schemes, I related two power series obtained by two different means. One of them was computed using geometric arguments via virtual localization methods and the other one came from working with representation theoretic objects called vertex algebras.

In this note, I give proof of the equality of the two power series by relying only on techniques related to Lagrange inversion. This makes my work on Quot schemes independent of the previous results in the literature and proves a new combinatorial identity.
\end{abstract}

\maketitle
\section{Introduction}
Grothendieck in his lecture \cite{FGA} introduced Quot schemes as a solution to a moduli problem which captures quotients of a fixed sheaf on a projective variety $X$. More explicitly: fixing an algebraic K-theory class $\beta$ and a sheaf $E$ on $X$, the scheme $\Quot(E,\beta)$ parametrizes the quotients
\[E\twoheadrightarrow F,\]
\sloppy where $ \ch(F) = \beta$. When $X$ is a surface, $E$ is a vector bundle, and $\beta$ has 1-dimensional support, $\Quot(E,\beta)$ admits a perfect obstruction theory as in \cite{BF} which was observed by Marian--Oprea--Pandharipande \cite{MOP1}. As such, one can integrate cohomology classes on $\Quot(E,\beta)$ with respect to its virtual fundamental class $\Big[\Quot(E,\beta)\Big]^{\textnormal{vir}}$. The resulting invariants can be combined into generating series by summing over the Euler characteristics $\chi(\beta)$ and were studied in \cite{AJLOP, OP1, JOP} and \cite{bojko3}. In the last reference, I focused on the case when $\beta=np$ for $p$ the class of a sky-scraper sheaf on $X$ and $n\geq 0$ is an integer.

The series of \textit{tautological invariants} studied in \cite{bojko3} take the form
\[	Z_{E}(f,g,\alpha;q) = \sum_{n\geq 0}q^n\int_{[\QuotX]^{\vir}} f(\alpha^{[n]})g(T^{\vir}),\]
which depends on the additional choice of a K-theory class $\alpha\in K^0(X)$ and multiplicative genera determined by their invertible power series $f,g\in \K\llbracket t\rrbracket$. To stay brief, I refer to \cite[§1.3]{bojko3} for a more thorough definition of $\alpha^{[n]}$, $f,g$, and $Z_E(f,g,\alpha;q)$ as an integral over the virtual fundamental class. One of the main results in \cite{bojko3} is the explicit expression for this generating series of invariants (see \cite[(4.23)]{bojko3}):
\begin{align*}
\label{Eq:ZESexp}
Z_{E}(f,g,\alpha;q)=
&\prod_{k_1\neq k_2}g(H_{k_1}-H_{k_2})^{c_1(X)^2}\prod_{k=1}^e\big(g(H_k)g(-H_k)\big)^{-ec_1(X)^2}\\&\prod_{k=1}^e\bigg(\frac{f(H_k)}{f(0)}\bigg)^{c_1(\alpha)\cdot c_1(X)}\prod_{k=1}^e\bigg(\frac{f(H_k)}{f(0)}\bigg)^{\frac{a}{e}c_1(E)c_1(X)}G_e(g^ef^a)^{c_1(X)^2},
\numberthis
\end{align*}
where $c_1(-)$ denotes the first Chern class, $e=\rk(E)$, $a=\rk(\alpha)$, and $G_e(R)$ is the power series \eqref{eqGeR} depending on $R(t)\in K\llbracket t\rrbracket$, and $H_k$ are the $e$ different Newton--Puiseux solutions to 
\[H^e = qg^e(H)f^a(H).\]
 When $E$ is a trivial vector bundle, a similar result was computed in \cite[Theorem 17]{AJLOP}. Comparing the terms under the exponent $(-)^{c_1(X)^2}$ led to a prediction for a more elegant form of \eqref{eqGeR} noted down in \eqref{eqExp} which is based on \cite[(17)]{AJLOP}. When $f,g$ are known, the latter form of $G_e(g^ef^a)$ is far more manageable in concrete computations. This made it possible for me to study the structural behaviour of \eqref{Eq:ZESexp} leading for example to the proof of a new case of the Segre--Verlinde correspondence  \footnote{See \cite{Johnson}, \cite{MOP2}, \cite{MOPhigher} for the origin of the Segre--Verlinde correspondence.} in \cite[Theorem 1.6]{bojko3}, a new type of symmetry based on \cite[§1.9]{AJLOP}, and the proof of rationality in \cite[Theorem 1.4]{bojko3} of the series \eqref{Eq:ZESexp} for special choices of $f,g$ including the $\chi_{y}$-genus paired with an exterior power of $\alpha^{[n]}$.

While the previous authors used methods relying on a virtual localization formula of \cite{PG} combined with a torus-action on $E=\CC^{e}$, I have instead utilized the wall-crossing framework of Joyce \cite{JoyceWC} and the methods of its application developed in \cite{bojko2, bojko3}. A further reason for proving Theorem \ref{thmCtC} was to make my work self-contained and to make it reliant solely on wall-crossing. This highlights how powerful our approach is, as I circumvent the need to apply localization along $E$ thus removing the restriction on what $E$ needs to be. Simultaneously, I am still able to give a closed formula for all integrals $Z_E(f,g,\alpha;q)$.

\begin{theorem}
\label{thmCtC}
Let $R(t)\in \K\llbracket t\rrbracket$ be a power series over a field $\K$ not involving $q$ and $R(0)\neq 0$.
Suppose that $H_k(q)$ are the $e>0$ different Newton--Puiseux solutions to $
H_k^e(q) = qR\big(H_k(q)\big).
$
Set
\begin{equation}
\label{eqGeR}
 G_e(R)= \exp\bigg[-\sum_{\begin{subarray}a n,m>0\\ j>0\end{subarray}}j\frac{1}{m}[z^{me+j}]\Big\{R^m(z)\Big\}\frac{1}{n}[z^{ne-j}]\Big\{R^n(z)\Big\}q^{n+m}\bigg]   
\end{equation}
then the following holds:
\begin{align*}
\label{eqExp}
&G_e(R)=  \prod_{k=1}^e\bigg(\frac{R(H_k)}{R(0)}\bigg)\prod_{k=1}^eH_k^{e}
\cdot \prod_{k_1\neq k_2 }\bigg(\frac{1}{H_{k_2}-H_{k_1}}\bigg)\prod_{k=1}^e\bigg(\frac{e}{H_k}-\frac{R'(H_k)}{R(H_k)}\bigg).
\numberthis
\end{align*}
\end{theorem}
\begin{remark}
    The result stated in \eqref{Eq:ZESexp} was computed using the vertex algebra constructed by Joyce \cite{Joycehall} on the homology of the moduli stack $\mM_X$ of all perfect complexes on $X$. This induced a Lie bracket on a quotient of the homology where the classes $[\Quot(E,\beta)]^{\vir}$ can be naturally transported. By the work of \cite{gross}, there is an explicit description of this vertex algebra and the resulting Lie algebra. This was used in \cite{bojko3} to write down a closed expression for $[\Quot(E,\beta)]^{\vir}$ in terms of a basis of the homology of $\mM_X$ as a consequence of an iterated application of the Lie bracket. Using the more general result, $Z_E(f,g,\alpha;q)$ could be computed directly by simple manipulation of exponential series and second-order derivation. After applying a generalization for Newton--Puiseux series of the Lagrange inversion based on \cite{Gessel}, I obtained the expression \eqref{Eq:ZESexp}.

The techniques used by \cite{AJLOP} are orthogonal to ours. After applying the virtual localization of \cite{PG} for the $\GG_m^e$ action on $E=\CC^e$, they reduce the computation of $Z_E(f,g,\alpha;q)$ to integrals over multiple copies of $\textnormal{Hilb}^n(X)$. A big part is also played by the (virtual) normal contributions coming from the direction along which each $\GG_m$ acts. This feeds into the formula \eqref{eqExp} after an application of the multivariable Lagrange inversion formula from \cite{Gesselmulti}. As mentioned previously, the localization method is limited to situations when there is a $\GG_m^e$ action on $E$, while our approach works for any $E$. In fact, I stated the results in \cite{bojko3} for any torsion-free sheaf $E$ on a surface $X$.
\end{remark}
 The idea of the proof of Theorem \ref{thmCtC} was motivated by the case when $e=1$ which restricts to a unique formal power series $H(q)$ and thus makes the expression considerably simpler. I am grateful to the MathOverflow authors Alex Gavrilov and esg who proved this special case in \cite{MOurl}.
\begin{corollary}[\cite{MOurl}]
\label{cor:onevar}
Let $R(t)\in \K\llbracket t\rrbracket$ be the power series from Theorem \ref{thmCtC} and $H(q)$ the unique power series solution to $
H(q) = qR\big(H(q)\big).
$
Set
\begin{equation}
\label{eqGeRs}
 G(R)= \exp\bigg[-\sum_{\begin{subarray}a n,m>0\\ j>0\end{subarray}}j\frac{1}{m}[z^{m+j}]\Big\{R^m(z)\Big\}\frac{1}{n}[z^{n-j}]\Big\{R^n(z)\Big\}q^{n+m}\bigg]   
\end{equation}
then the following holds:
\begin{align*}
\label{eqExps}
&G(R)=  \bigg(\frac{R(H)}{R(0)}\bigg)
\cdot \bigg(1-\frac{R'(H)}{R(H)}H\bigg).
\numberthis
\end{align*}
\end{corollary}
This simpler version is related to the Lagrange inversion, various versions of which are neatly summarized in \cite[Theorem 2.1.1]{Gessel}. Yet the result of Corollary \ref{cor:onevar} can not be deduced as a direct consequence of the known variations of it. Instead, one needs to use a trick by involving an additional variable $p$ and then taking a limit as $p\to q$. Alternatively, one could compare \eqref{eqGeRs} and \eqref{eqExps} by working with contour integrals and taking residue. 

Moving on to the multivariable Lagrange inversion, one might hope to use one of the results appearing in \cite{Gesselmulti}. Unfortunately, there is no natural way of doing so. Additionally, as the increased complexity of the multivariable expression suggests in \eqref{eqHLG}, there is no immediate way of generalizing Corollary \ref{cor:onevar} to the full statement. Instead, I was forced to adapt the proofs by including Newton--Puiseux series and computing with roots of unity. 
\section{Lagrange inversion for Newton--Puiseux generating series}

The primary tool I will rely on is the following slightly less standard Lagrange inversion formula collecting multiple results of Gessel \cite{Gessel}:
\begin{lemma}[Gessel {\cite[Thm. 2.1.1, eq. (2.2.8), (2.2.9)]{Gessel}} ]
\label{ThGes}
Let $R(t)=\sum_{n=0}^{\infty}r_nt^n$ be a power series not involving $q$ and satisfying $r_0=1$, then for the unique solution $H(q)$ satisfying $H(q) = qR\big(H(q)\big)$ and a formal Laurent series $\phi=\phi(t)$, I have
\begin{align}
\label{eqLag}
\phi\big( H(q)\big) &= [t^0]\big\{\phi(t)\big\}+[t^{-1}]\big\{\phi'(t)\textnormal{log}\big(R(t)\big)\big\} +\sum_{n\neq 0}\frac{1}{n}[t^{n-1}]\big\{\phi'(t)R^n(t)\big\} q^n,\\
\log\big(R(H(q))\big)&= \textnormal{log}(H/q) = \sum_{m>0}\frac{1}{m}[t^{m}]\Big\{R^m(t)\Big\}q^m.
\end{align}
\end{lemma}
Note that I can easily remove the assumption $r_0 = 1$ by working with $R(t)/r_0$ instead. This will be done at the end to recover Theorem \ref{thmCtC} in which case I always implicitly assume that 
\[r_0  =R(0)\neq 0.\]

By looking at Theorem \ref{thmCtC}, it is clear that I need a modification of this result to include multiple Newton--Puiseux solutions. For the definition of Newton--Puiseux solutions for implicit equations see for example \cite{NewPui} where they are defined as power series in $\K\llbracket x^{\frac{1}{e}} \rrbracket$. A similar adaptation appeared already in \cite[Lemma 4.12]{bojko2}. Note that below, $e$ will sometimes be used to denote the Euler number but it is always clear when this is the case, so I do not introduce a new notation.

\begin{corollary}
\label{CorMG}
Working with the power series $R(t)$ from Lemma \ref{ThGes}, let $H_k(q)$ for $k=1,\ldots,e$ be the different Newton--Puiseux series which are solutions to \begin{equation}
\label{eqGSo}H^e_k(q)=qR\big(H_k(q)\big),\end{equation}
then for any Laurent series $\phi(t) = \sum_{n \in \Z}\phi_{n}t^n$, I have
\begin{align*}
\label{eqHLG}
\sum_{k=1}^e\phi\big(H_k(q)\big) &= e\phi_0+ [t^{-1}]\big\{\phi'(t)\textnormal{log}\big(R(t)\big)\big\}\\
&+ \sum_{n\neq 0}\frac{1}{n} [t^{ne-1}]\Big\{\phi'(t)R^n(t)\Big\}q^n,\\
\sum_{k=1}^e\log\Big(R\big(H_k(q)\big)\Big) &= \textnormal{log}\Big(\Big(\prod_{k=1}^e H_k(q)\Big)/q\Big) = \sum_{n>0}\frac{1}{n}[t^{ne}]\Big\{R^n(t)\Big\}q^n.
\numberthis
\end{align*}
\end{corollary}
\begin{proof}
Let $\tilde{H}(q^{\frac{1}{e}})$ be the unique Newton--Puiseux series satisfying $$\tilde{H}(q^{\frac{1}{e}})=q^{\frac{1}{e}}R^{\frac{1}{e}}\Big(\tilde{H}\big(q^{\frac{1}{e}} \big)\Big)$$ for a fixed $e$th root of $R$. Then the $e$ different  solutions of \eqref{eqGSo} can be expressed as $H_k(q) = \tilde{H}\big(e^{\frac{2\pi i k}{e}}q^{\frac{1}{e}}\big)$ for $k=1,\cdots, e$ because they solve the $e$ different roots of it:
\[H_k(q) = e^{\frac{2\pi i k}{e}}q^{\frac{1}{e}}R^{\frac{1}{e}}\Big(H_k\big(q \big)\Big).\]

I obtain \begin{align*} \sum_{k=1}^e\phi\big(H_k(q)\big)=&\sum_{k=1}^e\Big([t^0]\big\{\phi(t)\big\}+[t^{-1}]\big\{\frac{\phi'(t)}{e}\log\big(R(t)\big)\big\}\\
+&\sum_{n\neq 0}\frac{1}{n}[t^{n-1}]\big\{\phi'(t)R^{\frac{n}{e}}(t)\big\}q^{\frac{n}{e}}e^{\frac{2\pi i k n}{e}}\Big),\end{align*}
which gives the required result \eqref{eqHLG}. The second equation follows by a similar argument.
\end{proof}

The first version of the proof will only use the above results and some tricks for generating series.
\begin{proof}[Formal power series proof of Theorem \ref{thmCtC} inspired by esg \cite{MOurl}]
By using Corollary \ref{CorMG}, while working with power series over the field $\K(\!(p)\!)$, I can write
\begin{align*}
&\sum_{j>0}\sum_{n,m>0}j[t^{en+j}]\{R^n(t)\}[t^{em-j}]\{R^m(t)\}\frac{p^n}{n}\frac{q^m}{m}\\=& -\sum_{j>0}\frac{1}{j}\Big(\sum_{k_1=1}^eH_{k_1}^{-j}(p)\Big)\Big(\sum_{k_2=1}^eH^{j}_{k_2}(q)\Big) \\
-&\sum_{j>0}\sum_{-j\leq ne<0}[t^{en+j}]\{R^n(t)\}\frac{p^n}{n}\Big(\sum_{k=1}^eH_k^j(q)\Big)\\-&\sum_{j>0}\sum_{k=1}^e[t^{-1}]\big\{t^{-j-1}\log\Big(R(t)\Big)\big\}H_k^j(q) \\
=&\sum_{k_1,k_2}\log\Big(1-\frac{H_{k_1}(q)}{H_{k_2}(p)}\Big)+ \sum_{n>0}\sum_{j\geq ne}[t^{j-en}]\big\{R^{-n}(t)\big\}\Big(\sum_{k=1}^e H^j_k(q)\Big)\frac{p^{-n}}{n}\\-&\sum_{k=1}^e\log\Big(R(H_k(q))\Big)\\
=&\sum_{k_1,k_2}\log\Big(1-\frac{H_{k_1}(q)}{H_{k_2}(p)}\Big)
- \sum_{n>0}\sum_{k=1}^e\frac{1}{n}\frac{H_k^{en}(q)}{R^{n}\big(H_k(q)\big)p^n}\\-&\sum_{k=1}^e\log\Big(R(H_k(q))\Big)\\
=&\sum_{k_1,k_2}\log\Big(1-\frac{H_{k_1}(q)}{H_{k_2}(p)}\Big)
- e\textnormal{log}\Big(1-\frac{q}{p}\Big)\\-&\sum_{k=1}^e\log\Big(R(H_k(q))\Big),
\end{align*}
where in the last step, I used \eqref{eqGSo} applied to the variable $q$ with coefficients in the field of Laurent series in $p$. Therefore, I obtain after taking exponential and taking the limit $p\to q$:
\begin{align*}
&\lim_{p\to q}\prod_{k_1\neq k_2}\big(H_{k_1}(p)-H_{k_2}(q)\big)\prod_{k=1}^e\frac{H_k(p)-H_k(q)}{p-q}\frac{1}{R\big(H_k(q)\big)}\prod_{k=1}^e\Big(\frac{p}{H_k(p)}\Big)^e\\
=&\prod_{k_1\neq k_2}\big(H_{k_1}(q)-H_{k_2}(q)\big)q^e\prod_{k=1}^e\Big(\frac{dq}{dH_k(q)}\Big)^{-1}\prod_{k=1}^e\frac{1}{H^e_k(q)R\big(H_k(q)\big)}\\
=&\prod_{k_1\neq k_2}\big(H_{k_1}(q)-H_{k_2}(q)\big)\prod_{k=1}^e\frac{1}{H^e_k(q)R\big(H_k(q)\big)}\Big(\frac{e}{H_k}-\frac{R'\big(H_k(q)\big)}{R\big(H_k(q)\big)}\Big)^{-1}.
\end{align*}
After taking the reciprocal of the final expression, I obtain the right-hand side of \eqref{eqExp} if $R(0) = 1$. I thus only need to replace $R(t)$ by $R(t)/R(0)$  to conclude the more general result.
\end{proof}
An alternative way of showing this result is by starting with
 \begin{align*}
 \label{eq:firststep}
 -&\sum_{\begin{subarray}a n,m>0\\ j>0\end{subarray}}j\frac{1}{m}[z^{m+j}]\Big\{R^m(z)\Big\}\frac{1}{n}[z^{n-j}]\Big\{R^n(z)\Big\}q^{n+m}\\ = &\sum_{j>0}\frac{1}{j}\Big(\sum_{n>0}[q^n]\Big\{\sum_{k_1=1}^eH_{k_1}^{-j}(q)\Big\}\Big)\Big(\sum_{k_2=1}^e H^{j}_{k_2}(q)\Big)
 \numberthis
 \end{align*} which used Corollary \ref{CorMG}. I then use contour integration combined with residue computation to write the last expression as the logarithm of \eqref{eqExp}.
\begin{proof}[Analytic proof inspired by Alex Gavrilov \cite{MOurl}]
Set the notation 
\[H^j(q) := \sum_{k=1}^eH^j_k(q),\]
and assume without loss of generality that $H^j(q)$ are all analytic. When $\phi(q) =\sum_{n\in \Z}\phi_n q^n$ is a Laurent series, I will write
\[
[q^{>0}]\big\{ \phi(q)\big\} := \sum_{n>0}\phi_nq^{n}.
\]
Then the following identity holds:
\begin{align*}
[q^{>0}]\big\{H^{-j}(q)\big\}&:=\sum_{n>0}[z^n]\big\{H^{-j}(z)\big\}q^n = \frac{1}{2\pi i} \oint_{|z|=\rho}\sum_{n>0}\Big(\frac{q}{z}\Big)^n\frac{1}{z}H^{-j}(z)dz\\
 &= \frac{1}{2\pi i} \oint_{|z|=\rho}\Big(\frac{1}{z-q}-\frac{1}{z}\Big)H^{-j}(z)dz,
\end{align*}
which uses that 
\[
\frac{1}{2\pi i}\oint_{|z|=\rho} z^{j}dz = \begin{cases}
1&\textnormal{if}\quad j={-1},\\
0&\textnormal{otherwise}.
\end{cases}
\]
Plugging this into \eqref{eq:firststep} leads to
\begin{align*}
&\sum_{\begin{subarray}a j>0\end{subarray}}\frac{1}{j}H^j(q)[q^{>0}]\big\{H^{-j}(q)\big\}\\=&\frac{1}{2\pi i}\oint_{|z|=\rho}\Big(\frac{1}{z-q}-\frac{1}{z}\Big)\sum_{k_1,k_2}\log\Big(\frac{H_{k_1}(z)}{H_{k_1}(z)-H_{k_2}(q)}\Big)dz,
\end{align*}
where I view the factor consisting of the logarithm as formal until I sum over $k_1,k_2$ when it becomes a function in $q$  instead of its roots. Additionally, I always expand the logarithm as if $|H_{k_1}(z)| > |H_{k_2}(q)|$ for any $z,q$ such that $|q|<|z|=\rho$. This is the case when $\rho$ is sufficiently small, because $H_k(0) = 0$ and the leading coefficient of $H_k(q)$ is $q^{\frac{1}{e}}$. The resulting integral is well-defined but it can not be computed using residues because the integrand is not a meromorphic function in $B_\rho(0)$ - the disc of radius $\rho$.
Note however that
\begin{align*}
&\frac{1}{2\pi i}\oint_{|z|=\rho}\Big(\frac{1}{z-q}-\frac{1}{z}\Big)\log\Big(\frac{z}{z-q}\Big)dz\\
&= -\int_{0}^{1}\Big(\frac{\frac{r}{\rho}e^{2\pi i(\theta-\tau)}}{1-\frac{r}{\rho}e^{2\pi i(\theta-\tau)}}\Big)\log\Big(1-\frac{r}{\rho}e^{2\pi i(\theta-\tau)}\Big)d\tau=0,
\end{align*}
where I used the substitution $z=\rho e^{2\pi i \tau}, q = r e^{2\pi i\theta}$ and the integral vanishes because it is proportional to the integral of a total derivative of a $2\pi$ periodic function. 

Therefore, I may work with
\begin{align*}
\label{eqInt}
\frac{1}{2\pi i}\int _{|z|=\rho}&D(z,q)dz\\
=&\frac{1}{2\pi i}\oint_{|z|=\rho}\Big(\frac{1}{z-q}-\frac{1}{z}\Big)\sum_{k_1,k_2}\log\Big(\frac{H_{k_1}(z)\big(z-q\big)^{\frac{1}{e}}}{z^{\frac{1}{e}}\big(H_{k_1}(z)-H_{k_2}(q)\big)}\Big)dz
\numberthis
\end{align*}
 instead. Here, the powers $(-)^{\frac{1}{e}}$ disappear after taking the sum and are only meant for notational purposes. The integrand $D(z,q)$ is a meromorphic function in $z$ with poles of order 1 at $z=0$ and $z=q$. The absence of further singularities is a consequence of choosing $\rho$ sufficiently small such that $\prod_{i}H_i(z)$ is univalent in $B_{\rho}(0)$, and so is $\prod_{i,j}\big(H_i(z)-H_j(q)\big)$ for any $q\in B_{\rho}(0)$. The precise reason for choosing the extra factor is the expression
 \[
\prod_{k_1,k_2}\frac{H_{k_1}(z)}{H_{k_1}(z)-H_{k_2}(q)} = \Big(\frac{z}{z-q}\Big)^e\prod_{k_1,k_2}\frac{1+\mO(z^{\frac{1}{e}})}{1+\mO(q^{\frac{1}{e}})}
\]
which can be also used to conclude that the sum in \eqref{eqInt} has no singularities for small enough $q,z$.
After reorganizing the logarithms slightly, the integral \eqref{eqInt} can be expressed as the sum of the following two residues:
\begin{align*}
    \textnormal{Res}_{z=0}\big(D(z,q)\big) &= \lim_{z\to 0}\Big\{\Big(\frac{z}{z-q}-1\Big)\Big[\sum_{k=1}^e\log\Big(\frac{H_k(z)(z-q)^{\frac{1}{e}}}{z^{\frac{1}{e}}\big(H_k(z)-H_k(q)\big)}\Big)\\
    &+\sum_{k_1\neq k_2}\log\Big(\frac{H_{k_1}(z)(z-q)^{\frac{1}{e}}}{z^{\frac{1}{e}}\big(H_{k_1}(z)-H_{k_2}(q)\big)}\Big)\Big]\Big\}\\
    &=\sum_{k=1}^e\log\Big(\frac{H_k^e(q)}{q}\Big)\\
     \textnormal{Res}_{z=q}\big(D(z,q)\big) &= \lim_{z\to q}\Big\{\Big(1-\frac{z-q}{z}\Big)\Big[\sum_{k=1}^e\log\Big(\frac{H_k(z)(z-q)}{z(H_k(z)-H_k(q))}\Big)\\
     &+\sum_{k_1\neq k_2}\log\Big(\frac{H_{k_1}(z)}{H_{k_1}(z)-H_{k_2}(q)}\Big)\Big]\Big\}\\
     &=\sum_{k=1}^e\log\Big[\frac{H_k(q)}{q}\frac{dq}{d H_k(q)}\Big] + \sum_{k_1\neq k_2}\log\Big[\frac{H_{k_1}(q)}{\big(H_{k_1}(q)-H_{k_2}(q)\big)}\Big]
\end{align*}
After exponentiating, this implies \eqref{eqExp}. Finally, note that as the equation is polynomial in coefficients in each degree, the statement holds even after removing the analyticity condition on $H^j(q)$.
\end{proof}
In conclusion, I was able to prove in two different ways the equality of two formal power series: the series in \eqref{eqGeR} which appeared naturally during a computation in my work \cite{bojko3}, and \eqref{eqExp} obtained by Arbesfeld--Johnson--Lim--Oprea--Pandharipande \cite{AJLOP}. Both of them are used to express the generating series of enumerative invariants \eqref{Eq:ZESexp} related to Quot schemes. This establishes the method used in \cite{bojko3} as an autonomous approach to study the structure of \eqref{Eq:ZESexp} and address open questions related to it.


\begin{ack}
I would like to thank A. Gavrilov, A. Mellit and S. Yurkevich. I am also grateful to Academia Sinica where I stayed while finishing work on this submission. 
\end{ack}

\begin{funding}
I was supported by ERC-2017-AdG-786580-MACI. This project has received funding from the European Research Council (ERC) under the European Union Horizon 2020 research and innovation program (grant agreement No 786580). 
\end{funding}


\bibliographystyle{emss.bst}
\bibliography{mybib.bib} 

\begin{thebibliography}{10}
\providecommand{\url}[1]{\texttt{#1}}
\providecommand{\urlprefix}{URL }
\providecommand{\eprint}[2][]{\url{#2}}

\bibitem{AJLOP}
N.~Arbesfeld, D.~Johnson, W.~Lim, D.~Oprea, and R.~Pandharipande, The virtual {$K$}-theory of {Q}uot schemes of surfaces. \emph{J. Geom. Phys.} \textbf{164} (2021), Paper No. 104154, 36 \Zbl{1476.14010} \MR{4220750}

\bibitem{NewPui}
E.~R.~G. Barroso, P.~D.~G. P{\'{e}}rez, and P.~Popescu-Pampu, Variations on inversion theorems for newton{\textendash}puiseux series. \emph{Mathematische Annalen} \textbf{368} (2016), no. 3-4, 1359--1397 \Zbl{1386.14018} \MR{3673657}

\bibitem{BF}
K.~Behrend and B.~Fantechi, The intrinsic normal cone. \emph{Inventiones Mathematicae} \textbf{128} (1997), 45–88

\bibitem{bojko3}
A.~Bojko, Wall-crossing for punctual quot-schemes. 2021, \eprint{2111.11102}

\bibitem{bojko2}
A.~Bojko, Wall-crossing for zero-dimensional sheaves and hilbert schemes of points on calabi--yau 4-folds. 2021, \eprint{2102.01056}

\bibitem{Gesselmulti}
I.~M. Gessel, A combinatorial proof of the multivariable {L}agrange inversion formula. \emph{J. Combin. Theory Ser. A} \textbf{45} (1987), no.~2, 178--195 \Zbl{0651.05009} \MR{894817}

\bibitem{Gessel}
I.~M. Gessel, Lagrange inversion. \emph{J. Combin. Theory Ser. A} \textbf{144} (2016), 212--249 \MR{3534068}

\bibitem{PG}
T.~Graber and R.~Pandharipande, Localization of virtual classes. \emph{Invent. Math.} \textbf{135} (1999), no.~2, 487--518 \Zbl{1491.14078} \MR{1666787}

\bibitem{gross}
J.~Gross, The homology of moduli stacks of complexes. 2019, \eprint{1907.03269}

\bibitem{FGA}
A.~Grothendieck, \emph{Fondements de la g\'{e}om\'{e}trie alg\'{e}brique. [{E}xtraits du {S}\'{e}minaire {B}ourbaki, 1957--1962.]}. Secr\'{e}tariat math\'{e}matique, Paris, 1962 \Zbl{0239.14002} \MR{0146040}

\bibitem{Johnson}
D.~Johnson, Universal series for {H}ilbert schemes and strange duality. \emph{Int. Math. Res. Not. IMRN}  (2020), no.~10, 3130--3152 \Zbl{1440.14019} \MR{4098636}

\bibitem{JOP}
D.~Johnson, D.~Oprea, and R.~Pandharipande, Rationality of descendent series for {H}ilbert and {Q}uot schemes of surfaces. \emph{Selecta Math. (N.S.)} \textbf{27} (2021), no.~2, Paper No. 23, 52 \Zbl{1461.14008} \MR{4244329}

\bibitem{Joycehall}
D.~Joyce, Ringel–hall style lie algebra structures on the homology of moduli spaces. 2019, \urlprefix\url{https://people.maths.ox.ac.uk/joyce/hall.pdf}, visited on 16 November 2023

\bibitem{JoyceWC}
D.~Joyce, Enumerative invariants and wall-crossing formulae in abelian categories. 2021, \eprint{2111.04694}

\bibitem{MOP1}
A.~Marian, D.~Oprea, and R.~Pandharipande, Segre classes and {H}ilbert schemes of points. \emph{Ann. Sci. \'{E}c. Norm. Sup\'{e}r. (4)} \textbf{50} (2017), no.~1, 239--267 \Zbl{1453.14016} \MR{3621431}

\bibitem{MOP2}
A.~Marian, D.~Oprea, and R.~Pandharipande, The combinatorics of {L}ehn's conjecture. \emph{J. Math. Soc. Japan} \textbf{71} (2019), no.~1, 299--308 \Zbl{1422.14008} \MR{3909922}

\bibitem{MOPhigher}
A.~Marian, D.~Oprea, and R.~Pandharipande, Higher rank {S}egre integrals over the {H}ilbert scheme of points. \emph{J. Eur. Math. Soc. (JEMS)} \textbf{24} (2022), no.~8, 2979--3015 \Zbl{1495.14006} \MR{4416594}

\bibitem{MOurl}
Mathoverflow-users, Comparing two power-series. 2021, \urlprefix\url{https://mathoverflow.net/questions/397619/comparing-two-power-series/398435?noredirect=1#comment1024969_398435}, visited on 16 November 2023

\bibitem{OP1}
D.~Oprea and R.~Pandharipande, Quot schemes of curves and surfaces: virtual classes, integrals, {E}uler characteristics. \emph{Geom. Topol.} \textbf{25} (2021), no.~7 \Zbl{1505.14010} \MR{4372634}

\end{thebibliography}









\end{document}